\documentclass{amsart}
\usepackage{amssymb}
\usepackage{amsmath}
\usepackage{amscd}
\usepackage{url}
\usepackage{bbm}
\usepackage[draft=false]{hyperref}

\newtheorem{prop}{Proposition}[section]

\newtheorem{lem}[prop]{Lemma}
\newtheorem{thm}[prop]{Theorem}
\newtheorem{cor}[prop]{Corollary}
\theoremstyle{definition}

\newtheorem{remar}[prop]{Remark}

\newcommand{\Norm}{{\mathrm {Norm}}}

\newcommand{\tr}{{\mathrm {tr}}}

\newcommand{\spin}{{\mathrm {spin}}}
\newcommand{\Spin}{{\mathrm {Spin}}}

\newcommand{\Gal}{\mathrm {Gal}}

\newcommand{\A}{{\mathbb A}}
\newcommand{\CC}{{\mathbb C}}
\newcommand{\C}{{\mathbb C}}

\newcommand{\R}{{\mathbb R}}
\newcommand{\PP}{{\mathbb P}}
\newcommand{\QQ}{{\mathbb Q}}
\newcommand{\Q}{{\mathbb Q}}
\newcommand{\ZZ}{{\mathbb Z}}
\newcommand{\Z}{{\mathbb Z}}

\newcommand{\HH}{{\mathfrak H}}

\newcommand{\AAA}{{\mathcal A}}

\newcommand{\OOO}{{\mathcal O}}

\newcommand{\pp}{{\mathfrak p}}

\newcommand{\nn}{{\mathfrak n}}

\newcommand{\qq}{{\mathfrak q}}

\newcommand{\FF}{{\mathbb F}}

\newcommand{\GL}{\mathrm {GL}}

\newcommand{\SL}{\mathrm {SL}}
\newcommand{\Sp}{\mathrm {Sp}}
\newcommand{\JR}{\mathrm {JR}}

\newcommand{\GSp}{\mathrm {GSp}}

\newcommand{\Qbar}{\overline{\mathbb Q}}

\newcommand{\Fbar}{\overline{\mathbb F}}
\newcommand{\rhobar}{\overline{\rho}}

\newcommand{\dd}{{\mathfrak d}}

\newcommand{\kk}{{\mathbf k}}
\newcommand{\LL}{{\mathbf l}}

\newcommand{\divv}{\mathrm{div}}

\newcommand{\lmfdbecnf}[4]{\href
    {https://www.lmfdb.org/EllipticCurve/#1/#2/#3/#4}
    {\texttt{#2-#3#4}}}

\newcommand{\lmfdbhmf}[3]{\href
    {https://www.lmfdb.org/ModularForm/GL2/TotallyReal/#1/holomorphic/#1-#2-#3}
    {\texttt{#1-#2-#3}}}

\DeclareFontFamily{U}{wncy}{}
    \DeclareFontShape{U}{wncy}{m}{n}{<->wncyr10}{}
    \DeclareSymbolFont{mcy}{U}{wncy}{m}{n}
    \DeclareMathSymbol{\Sh}{\mathord}{mcy}{"58} 

\begin{document}
\title{Residual paramodularity of a certain Calabi-Yau threefold}
\author{Neil Dummigan}
\author{Gonzalo Tornaría}

\date{December 16th, 2024.}
\address{University of Sheffield\\ School of Mathematical and Physical Sciences\\
Hicks Building\\ Hounsfield Road\\ Sheffield, S3 7RH\\
U.K.}
\email{n.p.dummigan@shef.ac.uk}
\address{Centro de Matemática, Universidad de la República, Montevideo, Uruguay}
\email{tornaria@cmat.edu.uy}

\begin{abstract}
    \tolerance=500
    We prove congruences of Hecke eigenvalues between cuspidal Hilbert newforms $f_{79}$ and $h_{79}$ over $F=\Q(\sqrt{5})$, of weights $(2,2)$ and $(2,4)$ respectively, level of norm $79$. In the main example,  the modulus is a divisor of $5$ in some coefficient field, in the secondary example a divisor of $2$. 

The former allows us to prove that the $4$-dimensional mod-$5$ representation of $\Gal(\Qbar/\Q)$ on the $3^{\mathrm{rd}}$ cohomology of a certain Calabi-Yau threefold comes from a Siegel modular form $F_{79}$ of genus $2$, weight $3$ and paramodular level $79$. This is a weak form of a conjecture of Golyshev and van Straten. 

In aid of this, we prove also a congruence of Hecke eigenvalues between $F_{79}$ and the Johnson-Leung-Roberts lift $\JR(h_{79})$, which has weight $3$ and paramodular level $79\times 5^2$.
\end{abstract}



\maketitle

\section{Background and motivation} Golyshev and van Straten \cite[Theorem 2.5]{GvS} constructed a flat, projective family $f: \mathcal{Y}\rightarrow\PP^1$, defined over $\ZZ$, whose fibres $Y_t:=f^{-1}(t)$, for $t\notin\Sigma_C:=\{0,\infty, \frac{-11\pm 5\sqrt{5}}{32}\}$, are Calabi-Yau $3$-folds with Hodge-numbers $h^{0,3}=h^{1,2}=h^{2,1}=h^{3,0}=1$. It follows from their Theorem 2.12 that for any $t\in (\PP^1-\Sigma_C)(\Q)$, the representation of $\Gal(\Qbar/\Q)$ on $H^3_{\text{\'et}}(Y_t\otimes\Qbar, \FF_5)$ is induced from that of $\Gal(\Qbar/\Q(\sqrt{5}))$ on $E_{L, t/u}(\Qbar)[5](-1)$ (Tate twist of $5$-torsion), where $u=\frac{-11\pm 5\sqrt{5}}{2}$ and $E_{L, t/u}$ is the fibre at $t/u$ in a certain pencil of elliptic curves, isogenous to the standard Legendre family. The values of $u$ are among those giving singular fibres of the Ap\'ery family $E_A$ of elliptic curves with $5$-torsion, and $Y_t$ is a desingularisation of a kind of convoluted fibre product of $E_L$ and $E_A$, whose fibre over $s$ is the product of $E_{L, t/s}$ and $E_{A, s}$. 

When one counts points of $Y_t(\FF_p)$, values of $s$ for which $E_{A, s}$ is nonsingular, necessarily containing a point of order $5$, contribute a multiple of $5$, whereas $u=\frac{-11\pm 5\sqrt{5}}{2}$ make a non-trivial contribution mod $5$, depending on the number of points of the other factors $E_{L, t/u}$. Putting this together with the connection between numbers of points and traces of Frobenius elements, we have the rough idea behind the proof of \cite[Theorem 2.12]{GvS}. The fact that the Ap\'ery family is equivalent to the universal elliptic curve over $X_1(5)\simeq \PP^1$ is what makes the prime $\ell=5$ special.

Since $t/u\in\Q(\sqrt{5})$, $E_{L, t/u}$ is an elliptic curve over the real quadratic field $F:=\Q(\sqrt{5})$, and hence is modular, by a theorem of Freitas, Le Hung and Siksek \cite{FLS}. Thus it shares its $L$-function (and the associated $\ell$-adic representations of $\Gal(\Qbar/F))$) with a cuspidal Hilbert newform of weight $(2,2)$ for $\Gamma_0(\mathfrak{n})$, where $\mathfrak{n}$ is the conductor of $E_{t/u}$. We see how the quadratic nature of the singular values of $u$ for the Ap\'ery family leads to this occurrence of Hilbert modular forms ``in nature''.

In the special case $t=-1$, Golyshev and van Straten observe that these Galois
conjugate elliptic curves are those labelled
\lmfdbecnf{2.2.5.1}{79.1}a1 and \lmfdbecnf{2.2.5.1}{79.2}a2
in LMFDB \cite{LMF}, with conductors $-1\pm 4\sqrt{5}$ of norm $79$,
and that the associated Hilbert modular forms are
\lmfdbhmf{2.2.5.1}{79.1}{a} and \lmfdbhmf{2.2.5.1}{79.2}{a},
either of which they call $f_{79}$.
They also rename $Y_{-1}$ as $\mathbf{Y}_{79}$.
The dimension formulas of Ibukiyama \cite{Ibu,IbuKita}
show that the space $S_3(K(79))$, of weight $3$ Siegel cusp forms for the paramodular subgroup $K(79)$ of $\Sp_4(\Q)$, is $8$-dimensional.
It has a $7$-dimensional subspace of Gritsenko lifts associated with $S^-_4(\Gamma_0(79))$, and a unique (up to scaling) non-lift Hecke eigenform $F_{79}$ (for $T_p$ and $T_{1,p^2}$, for all primes $p$).
Golyshev and van Straten conjecture that $\mathbf{Y}_{79}$ is modular, to be precise, that its $L$-function $L(\mathbf{Y}_{79},s)$, arising from the $4$-dimensional $\ell$-adic representations of $\Gal(\Qbar/\Q)$ on $H^3_{\text{\'et}}(\mathbf{Y}_{79}\otimes\Qbar, \Q_{\ell})$, is the same as the spin $L$-function $L(s, F_{79}, \spin)$, arising from Hecke eigenvalues of $F_{79}$.

Choosing $\ell=5$ (for reasons explained in the first paragraph), it would suffice to show that the $5$-adic representation $\rho_{F_{79},5}$ of $\Gal(\Qbar/\Q)$ attached to $F_{79}$ \cite{W} is isomorphic to  that on $H^3_{\text{\'et}}(\mathbf{Y}_{79}\otimes\Qbar, \Q_5)$. As an interim goal, we aim to prove that the residual representation $\rhobar_{F_{79},5}$ is isomorphic to that on $H^3_{\text{\'et}}(\mathbf{Y}_{79}\otimes\Qbar, \FF_5)$.

We consider the space $S_{(2,4)}(\Gamma_0(\mathfrak{n}))$ of Hilbert cusp forms over $F$, of non-parallel weight $(2,4)$, with $\mathfrak{n}=(1+4\sqrt{5})$, of norm $79$.
Magma \cite{BCP} tells us that this space is $4$-dimensional. It is spanned by a Hecke eigenform $g_{79}$ with Hecke eigenvalues in $F$, and a Galois conjugacy class of Hecke eigenforms (any one of which we fix and call $h_{79}$), with Hecke eigenvalues in a cubic extension $F(h_{79})$.

We observed an apparent congruence of Hecke eigenvalues between the weight-$(2,4)$ Hecke eigenform $h_{79}$ and the weight-$(2,2)$ Hecke eigenform $f_{79}$, modulo some degree-$1$ divisor $\lambda$ of $(\sqrt{5})$ in $F(h_{79})$. If this congruence holds then the residual mod-$5$ (or mod-$\lambda$), $2$-dimensional representations of $\Gal(\Qbar/F)$ attached to $f_{79}$ and $h_{79}$, namely $\rhobar_{f_{79},5}$ and $\rhobar_{h_{79},\lambda}$, are isomorphic, as are the induced $4$-dimensional representations of $\Gal(\Qbar/\Q)$. In fact, $F(h_{79})$ is generated over $F$ by an element with minimal polynomial $x^3+(\sqrt{5}-13)x^2-32x-144\sqrt{5}+304$, and $(\sqrt{5})=\lambda^3$ is totally ramified, so $\lambda$ is the unique divisor of $(\sqrt{5})$ in $F(h_{79})$.

We have seen that $\mathrm{Ind}_{F}^{\Q}(\rhobar_{f_{79}, 5})$ is isomorphic to the representation of $\Gal(\Qbar/\Q)$ on $H^3_{\text{\'et}}(\mathbf{Y}_{79}\otimes\Qbar, \FF_5)$. Therefore to show that $\rhobar_{F_{79},5}$ is isomorphic to  $H^3_{\text{\'et}}(\mathbf{Y}_{79}\otimes\Qbar, \FF_5)$, assuming that $\rhobar_{f_{79},5}\simeq\rhobar_{h_{79},\lambda}$, it would suffice to show that $\rhobar_{F_{79},5}\simeq\mathrm{Ind}_{F}^{\Q}(\rhobar_{h_{79},\lambda})$.

Now the weight-$(2,4)$ Hilbert cuspidal eigenform $h_{79}$ has its Johnson-Leung-Roberts lift $\mathrm{JR}(h_{79})$, which is a cuspidal Hecke eigenform in $S_3(K(79\times 5^2))$ \cite{JR}, and has associated $4$-dimensional Galois representation $\rho_{\mathrm{JR}(h_{79}),\lambda}\simeq\mathrm{Ind}_{F}^{\Q}(\rho_{h_{79},\lambda})$. Therefore, to show that $\rhobar_{F_{79},5}\simeq\mathrm{Ind}_{F}^{\Q}(\rhobar_{h_{79},\lambda})$, it would suffice to prove a mod-$\lambda$ congruence of Hecke eigenvalues between $F_{79}\in S_3(K(79))$ and $\mathrm{JR}(h_{79})\in S_3(K(79\times 5^2))$. 

\begin{thm}\label{thm:congruenceJR}
Let $F=\Q(\sqrt{5})$, $\mathfrak{n}=(1+4\sqrt{5})$ of norm $79$, and $h_{79}\in S_{(2,4)}(\Gamma_0(\mathfrak{n}))$ a Hecke eigenform as above, with Johnson-Leung -Roberts lift $\JR(h_{79})\in S_3(K(79\times 5^2))$. Let $F_{79}\in S_3(K(79))$ be the non-lift Hecke eigenform as above. Then for the unique prime divisor $\lambda$ of $5$ in $F(h_{79})$ there is a mod-$\lambda$ congruence of Hecke eigenvalues (of $T_p$ and $T_{1,p^2}$ for all $p\neq 5$) between $F_{79}$ and $\JR(h_{79})$.    
\end{thm}

To complete the proof that $\rhobar_{F_{79},5}$ is isomorphic to  $H^3_{\text{\'et}}(\mathbf{Y}_{79}\otimes\Qbar, \FF_5)$, it remains to prove the assumption that $\rhobar_{f_{79},5}\simeq\rhobar_{h_{79},\lambda}$. This is equivalent to proving the mod-$\lambda$ congruence of Hecke eigenvalues between $f_{79}\in S_{(2,2)}(\Gamma_0(\mathfrak{n}))$ and $h_{79}\in S_{(2,4)}(\Gamma_0(\mathfrak{n}))$, which was observed experimentally.
\begin{thm}\label{main} Let $F=\Q(\sqrt{5})$, $\mathfrak{n}=(1+4\sqrt{5})$ of norm $79$, and $f_{79}\in S_{(2,2)}(\Gamma_0(\mathfrak{n}))$, $h_{79}\in S_{(2,4)}(\Gamma_0(\mathfrak{n}))$ Hecke eigenforms as above. Then for the unique prime divisor $\lambda$ of $5$ in $F(h_{79})$ there is a mod-$\lambda$ congruence of Hecke eigenvalues (of $T_{\pp}$ for all $\pp\neq (\sqrt{5})$) between $f_{79}$ and $h_{79}$.
\end{thm}
\begin{cor}\label{maincor} The $4$-dimensional representation of $\Gal(\Qbar/\Q)$ on $H^3_{\text{\'et}}(\mathbf{Y}_{79}\otimes\Qbar, \FF_5)$ is isomorphic to the mod-$5$ residual representation $\rhobar_{F_{79},5}$ attached to the non-lift Hecke eigenform $F_{79}\in S_3(K(79))$. Here $K(79)$ is a paramodular subgroup of $\Sp_4(\Q)$, and $\mathbf{Y}_{79}$ is a certain Calabi-Yau $3$-fold defined over $\Q$, as above.
\end{cor}
\begin{proof} We reiterate the preceding arguments. Since the Hecke eigenvalues in Theorem \ref{main} are traces of Frobenius on the irreducible representations $\rhobar_{f_{79}, 5}$ and $\rhobar_{h_{79}, \lambda}$ of $\Gal(\Qbar/F)$, we deduce that these representations are isomorphic, then that the induced representations $\mathrm{Ind}_{F}^{\Q}(\rhobar_{f_{79}, 5})$ and $\mathrm{Ind}_{F}^{\Q}(\rhobar_{h_{79}, \lambda})$ of $\Gal(\Qbar/\Q)$ are isomorphic. Similarly the congruence of Hecke eigenvalues in Theorem \ref{thm:congruenceJR} implies that the representations $\mathrm{Ind}_{F}^{\Q}(\rhobar_{h_{79}, \lambda})$ and $\rhobar_{F_{79}, 5}$ of $\Gal(\Qbar/\Q)$ are isomorphic. Hence $\mathrm{Ind}_{F}^{\Q}(\rhobar_{f_{79}, 5})$ is isomorphic to $\rhobar_{F_{79}, 5}$. But $\rhobar_{f_{79}, 5}\simeq E_{L, -1/u}[5](-1)$, and by \cite[Theorem 2.12]{GvS} $\mathrm{Ind}_{F}^{\Q}(E_{L, -1/u}[5](-1))$ is isomorphic to the representation of $\Gal(\Qbar/\Q)$ on $H^3_{\text{\'et}}(\mathbf{Y}_{79}\otimes\Qbar, \FF_5)$.
\end{proof}
\begin{remar} One checks using Magma that for $F=\Q(\sqrt{5})$ and $\mathfrak{n}$ of norm $31$, although there are analogous forms $f_{31}\in S_{(2,2)}(\Gamma_0(\mathfrak{n}))$ and $h_{31}\in S_{(2,4)}(\Gamma_0(\mathfrak{n}))$, they do not satisfy a congruence modulo a divisor of $5$. However, it is highly unlikely that any of the elliptic curves over $F$, in the isogeny class associated with $f_{31}$, is $E_{L, t/u}$ for a rational value of $t$. (With sufficient effort one could convert them all to Legendre form, and check each $\lambda u$.) 

    So this does not rule out the possibility that analogues of Theorems \ref{main} and \ref{thm:congruenceJR} hold for all non-zero rational values of $t$, not just $t=-1$. Indeed for $t=1$, computations of Euler factors for $\mathbf{Y}_{431}:=Y_1$ by D. van Straten, recently extended by N. Gegelia \cite{nutsa}, suggest that it is associated with a certain $F_{431}\in S_3(K(431))$, whose Hecke eigenvalues have been computed by G. Rama and the second-named author, using the quinary forms method of \cite{RT}. One observes experimentally a congruence between the relevant $f_{431}\in S_{(2,2)}(\Gamma_0(\mathfrak{n}))$
(labelled \lmfdbhmf{2.2.5.1}{431.1}{c} in \cite{LMF}, where now $\mathfrak{n}$ has norm $431$)
and a form $h_{431}\in S_{(2,4)}(\Gamma_0(\mathfrak{n}))$, modulo a divisor of $5$ in a coefficient field of degree $6$.
\end{remar}
\begin{remar} It is our hope that Corollary \ref{maincor} will be the first step in a full proof of modularity of $\mathbf{Y}_{79}$. It appears that at present, published modularity lifting theorems are not strong enough to deduce that the representation of $\Gal(\Qbar/\Q)$ on $H^3_{\text{\'et}}(\mathbf{Y}_{79}\otimes\Qbar, \Q_{5})$ is isomorphic to $\rho_{F_{79},5}$, the $5$-adic Galois representation attached to $F_{79}$.

The residual representation $\rhobar$ is irreducible, but its restriction to $\Gal(\Qbar/F)$ (and therefore to $\Gal(\Qbar/\Q(\zeta_5))$) is reducible, with $2$-dimensional composition factors isomorphic to $\rhobar_{h_{79, 5}}$ and its conjugate. It follows that the image of $\rhobar(\Gal(\Qbar/\Q(\zeta_5))$ is too small for the condition (v) of \cite[Theorem 9.1]{T} to be satisfied.

An anonymous referee suggested instead exploiting the reducibility of $\rhobar$ (when restricted to $\Gal(\Qbar/F)$, and therefore also to $\Gal(\Qbar/K)$, where $K$ is a suitable totally imaginary quadratic extension of $F$) to apply a theorem of Allen, Newton and Thorne \cite[Theorem 1.1]{ANT}. But the local $L$-factor at $79$ of $L(s, F_{79}, \spin)$ has degree $3$, by \cite[(2.13)]{S}, whereas if the ``Steinberg'' condition 6(c) of  \cite[Theorem 1.1]{ANT} were satisfied, it would be of degree $1$.
\end{remar}
\begin{remar} There are examples of Calabi-Yau threefolds $X/\Q$ such that $H^3(X)$ is $4$-dimensional and $L(X, s)$ is proved to be $L(s, \JR(h), \spin)=L(s, h)$, with $h$ a Hilbert modular form of weight $[2,4]$ for a real quadratic field \cite{DPS}, \cite[Theorem 4.1]{CSvS}. The example $\mathbf{Y}_{79}$ would not be like this, since $F_{79}$ is not a $\JR$-lift, only congruent mod $5$ to one.

Recent work of the second-named author and collaborators, to be reported on elsewhere, has produced examples of Calabi-Yau threefolds $X/\Q$ such that $H^3(X)$ is $4$-dimensional and $L(X, s)$ is proved to be $L(s, F, \spin)$ for some paramodular form $F$ of genus $2$ and weight $3$, not a lift of something simpler like a Hilbert modular form.
The proof uses the Faltings-Serre method as extended to
$\GSp_4(\ZZ_2)$ in \cite{BPPTVY}, which is not applicable to $\mathbf{Y}_{79}$, since it follows from the mod $32$ congruence in \cite[Example 8.3]{PY} that the mod $2$ residual representation would be reducible. (Note that whereas $\#\Sp_4(\FF_2)=720$, barely manageable, $\#\Sp_4(\FF_3)=51840$, too large to apply the method to $\ell=3$, without a new idea.)
\end{remar}
\medskip

Theorem \ref{thm:congruenceJR} is proved in \S 2, using the methods of \cite[\S 11]{DPRT}, \cite{RT}. In fact, this provides a beautiful illustration of the utility of the results of \cite{DPRT} in their full generality, for a paramodular level that is not squarefree ($79\times 5^2$ in this instance).

An outline of the proof of Theorem \ref{main} is provided in \S 3.
It suffices to check the congruence by computation for sufficiently many Hecke eigenvalues, using a Sturm bound argument involving intersection theory on a Hilbert modular surface, modelled on work of Burgos Gil and Pacetti \cite{BP}. Details of the proof are in \S 5, and the latter part of \S 4.

The fact that the characteristic $5$ is ramified in the field $F=\Q(\sqrt{5})$ is an additional technical difficulty, but the tools we need, namely generalised partial Hasse invariants and partial theta operators, are available thanks to work of Reduzzi and Xiao \cite{RX}, Diamond \cite{D1} and Diamond and Sasaki \cite{DS}. We introduce what we need in \S 4.

The proof of a similar congruence modulo a divisor of $2$ (Theorem \ref{mainmod2}) is also sketched, in \S 6. 

The authors declare that the data supporting the findings of this study are available at \cite{DT}.

\section{Proof of Theorem \ref{thm:congruenceJR}}
We do more than is strictly necessary to prove the theorem. Some of the computations of Hecke eigenvalues are purely for illustration.
\subsection{Hilbert modular forms for
\texorpdfstring{$\mathbf{\QQ(\sqrt5)}$}{ℚ(√5)}}
\label{sect:HMF}
Recall that $F=\Q(\sqrt{5})$, $\mathfrak{n}=(1+4\sqrt{5})$, of norm $79$, and that $S_{(2,4)}(\Gamma_0(\mathfrak{n}))$ is $4$-dimensional, spanned by a Hecke eigenform $g_{79}$ with Hecke eigenvalues in $F$, and a Galois conjugacy class of Hecke eigenforms (any one of which we fix and call $h_{79}$), with Hecke eigenvalues in a cubic extension $F(h_{79})$ of $F$, in which $(\sqrt{5})=\lambda^3$ is totally ramified. We also met the Hecke eigenform $f_{79}\in S_{(2,2)}(\Gamma_0(\mathfrak{n}))$, which has rational Hecke eigenvalues and, in fact, spans that space. 

Magma tells us that the eigenvalues of the Hecke operator $T_{\nn}$ acting on $g_{79}$ and $h_{79}$ are $-79$ and $+79$,
hence the eigenvalues for the Atkin-Lehner involution at $\nn$
are $+1$ and $-1$, respectively.
It also gives us the following Hecke eigenvalues, where for split primes we give only the norm and do not say which factor is which. For a prime ideal $\pp$, the Hecke eigenvalue of $T_{\pp}$ acting on a form $f$ is denoted $\mu_{\pp}(f)$. For $h_{79}$ and $g_{79}$ we look only at the reductions in $\FF_5$.

\smallskip
\begin{center}
\begin{tabular}{l|rrrrrrrrrrrrrrrrrrrrrr}
$\Norm(\pp)$ &
4& 5& 9& 11& 11& 19& 19& 29& 29& 31& 31& 41& 41& 49
\\[1pt]\hline\\[-9pt]
$\mu_\pp(f_{79})$ &
1& -2& -2& 0& -4& 8& 4& 6& -2& -8& 0& 2& -2& -2
\\[2pt]
$\mu_\pp(h_{79})\bmod{\lambda}$ &
1& 0& 3& 0& 1& 3& 4& 1& 3& 2& 0& 2& 3& 3
\\[2pt]
$\mu_\pp(g_{79})\bmod{\sqrt{5}}$ &
-1& 0& -3& 0& 1& -3& -4& -1& -3& 2& 0& 2& 3& -3
\end{tabular}
\end{center}
\smallskip


In the table we observe the congruence between $f_{79}$ and $h_{79}$ (for $\pp\neq(\sqrt{5})$) that is the subject of Theorem \ref{main}, but also numerical evidence for a mod-$\lambda$ congruence of Hecke eigenvalues between $h_{79}$ and a twist of $g_{79}$, specifically that $\mu_{\pp}(h_{79})$ and $\Norm(\pp)\mu_{\pp}(g_{79})$ are the same in $\FF_5$, for all prime divisors  $\pp\neq (\sqrt{5})$ in $F$. The twisting character is that associated with the quadratic extension $\Q(\zeta_5)$ of $F$, i.e. the restriction to $\Gal(\Qbar/F)$ of the mod $5$ cyclotomic character. The Magma code, and a much longer sequence of output, are easily located in \cite{DT}.


\subsection{The Johnson-Leung-Roberts lift}
\label{sect:JR}

Since $h_{79}\in S_{2,4}(\Gamma_0(\nn))$, its $\JR$-lift \cite{JR}
is a paramodular form $\JR(h_{79})\in S_3^{\text{new}}(K(79\times5^2))$.
The $\JR$-lift is characterized by the fact that
$L(s, h_{79})$ and
$L(s, \JR(h_{79}), \Spin)$ are the same. In particular, if $\sum a_n(\JR(h_{79}))n^{-s}:=L(s, \JR(h_{79}), \Spin)$, then
\[
    a_p(\JR(h_{79})) = \begin{cases}
        0
          & \text{if $p$ is inert in $F$;} \\
        \mu_{\pp_1}(h_{79}) + \mu_{\pp_2}(h_{79})
          & \text{if $p=\pp_1\pp_2$ is split in $F$; and} \\
        \mu_{(\sqrt5)}(h_{79})
          & \text{if $p=5$ is ramified in $F$.}
    \end{cases}
\]
When $p\neq 79$,
the Hecke eigenvalue $\nu_p(\JR(h_{79}))$ for $T_p$
is the same as 
$a_p(\JR(h_{79}))$.
The Atkin-Lehner eigenvalues
are $\varepsilon_5(\JR(h_{79}))=+1$ and $\varepsilon_{79}(\JR(h_{79}))=-1$.

 The following table lists the first eigenvalues for $T_p$ modulo $\lambda$,
 obtained from the data for $h_{79}$ in the previous section.
\smallskip
\begin{center}
\begin{tabular}{c|rrrrrrrrrrrrr}
$p$ &
2& 3& 5& 7& 11& 13& 17& 19& 23& 29& 31& 37& 41
\\[1pt]\hline\\[-9pt]
$\nu_p(\JR(h_{79}))\bmod{\lambda}$ &
 0&  0&  0&  0&  1&  0&  0&  2&  0&  4&  2&  0&  0
\end{tabular}
\end{center}
\smallskip



\subsection{The paramodular form of level 79}
\label{sect:para79}
As mentioned in the introduction,
the space $S_3(K(79))$
of weight 3 paramodular forms of level 79
is $8$-dimensional.
It has a $7$-dimensional subspace of Gritsenko lifts associated with
$S^-_4(\Gamma_0(79))$, and a unique (up to scaling)
non-lift Hecke eigenform $F_{79}$ (see \cite[Example 8.3]{PY}).
The following table lists the first
eigenvalues for $T_p$ and $T_{1,p^2}$
taken from \cite{Hein}.
A complete list of eigenvalues for $T_p$ with $p<1000$ can be found at
\cite{omf5}.

\smallskip
\begin{center}
\begin{tabular}{c|rrrrrrrrrrrrr}
$p$ &
2& 3& 5& 7& 11& 13& 17& 19& 23& 29& 31& 37& 41
\\[1pt]\hline\\[-9pt]
$\nu_p$ &
-5& -5& 3& 15& 26& -15& -60& 32& 50& 24& 142& -500& 240
\\
$\nu_{1,p^2}$ &
2& 4& -10& -24& 0& -158& -156& -712& -256& -988& -640& 2668& 876
\end{tabular}
\end{center}
\smallskip

Note that, as far as the table goes, if $p$ is inert in $F=\QQ(\sqrt5)$ then the Hecke eigenvalue $\nu_p(F_{79})$ is divisible by 5.
This is an indication of the congruence with $\JR(h_{79})$,
since $T_p$-eigenvalues of $\JR$-lifts for inert primes are always zero. Moreover, the eigenvalues $\nu_p(F_{79})$ for the split primes $p=11, 19, 29, 31$ and $41$ agree in $\FF_5$ with $\nu_p(\JR(h_{79}))=\mu_{\pp_1}(h_{79})+\mu_{\pp_2}(h_{79})$,
where $p=\pp_1\pp_2$.
The Atkin-Lehner eigenvalue is $\varepsilon_{79}(F_{79})=-1$,
which also agrees with $\varepsilon_{79}(\JR(h_{79}))$.


\subsection{Orthogonal modular forms of level
\texorpdfstring{79\texttimes5\textsuperscript2}{79 × 5²}}
In order to prove Theorem~\ref{thm:congruenceJR} we will use the
results of \cite{DPRT} to realize $\JR(h_{79})$ and $F_{79}$
in a specific space of orthogonal modular forms.

Following
\cite[§6]{DPRT}, let $D^-=79$ and $D^+=5^2$, and let $L$ be the
lattice as in \cite[Definition~6.7]{DPRT}. In particular, $L$ is a
positive definite quinary lattice of (half) discriminant $D=D^-D^+$.
Then \cite[Theorem 9.9, Theorem 10.1]{DPRT}
together prove that
\begin{equation}
\label{eq:DPRT}
    S_3^{79\text{-new},+-
    }(K(79\times5^2))_G
    \quad
    \simeq
    \quad
    M_{(0,0)}(\widehat{K(L)})_G
\end{equation}
Here the superscript $+-$ on the left hand side indicates that we
restrict to forms with $\varepsilon_5=+1$ and $\varepsilon_{79}=-1$,
and the subscript $G$ on both sides indicates that we restrict to
forms of ``general type''. On the right-hand side, $M_{(0,0)}(\widehat{K(L)})$ is a space of algebraic modular forms, with trivial coefficient system, for the adelic points of the orthogonal group of $L$. It is a space of functions on the finite set of classes in the genus of $L$. The open compact subgroup $\widehat{K(L)}$ is the stabiliser of $L\otimes\A_f$. 

From §\ref{sect:JR} we know that
$\JR(h_{79})$ appears in the left hand side of \eqref{eq:DPRT}.
From §\ref{sect:para79} we know that $F_{79}$ appears on the left hand
side of \eqref{eq:DPRT}. More specifically,
by \cite[7.5.6]{RS} the space of paramodular
oldforms of level $79\times 5^2$ corresponding to $F_{79}$ has
dimension 4.
This old space is splits into $\varepsilon_5=+1$ and $\varepsilon_5=-1$ eigenspaces of dimensions $3$ and $1$, respectively \cite[table on p.~186]{RS}. Hence we know that the left hand side of
\eqref{eq:DPRT} realizes $F_{79}$ with multiplicity $3$ (for all Hecke
operators outside $5$).

The (genus of the) positive definite quinary lattice $L$ can be
characterized by
\cite[Proposition 6.9, Theorem 5.14]{DPRT}.
It is a ``special'' lattice with $e_5(L)=+1$ and $e_{79}(L)=-1$
in the sense of \cite[Definition 5.1, 5.8]{DPRT}.
For a concrete realization, it is easy to verify that
the positive definite quinary quadratic form
\begin{multline*}
    Q(x_0,x_1,x_2,x_3,x_4) =  
    x_{0}^{2} + x_{0} x_{1} + x_{1}^{2} + x_{2}^{2} + 2 x_{3}^{2}
    + 5 x_{1} x_{4} + 5 x_{2} x_{4} + 5 x_{3} x_{4} + 100 x_{4}^{2}
\end{multline*}
satisfies these requirements.
Let $V=M_{(0,0)}(\widehat{K(L)})=M_{(0,0)}(Q)$, and let $T_2$ be the
$2$-neighbour operator acting on $V$.
Let $V_1=\ker(T_2+5)$ and $V_2=\ker T_2$.
For computational purposes, and for the reduction modulo 5
given by $\otimes\,\FF_5$,
we work with $\Z$-coefficients, though we still talk of subspaces and dimensions below.

\begin{lem}\mbox{}
\begin{enumerate}
    \item The dimension of $V$ is 612.
    \item The dimension of $V_1$ is 3.
    \item The dimension of $V_2$ is 6.
    \item The dimension of $V_1\otimes\FF_5\cap V_2\otimes\FF_5$
        is 1.
\end{enumerate}
\end{lem}
\begin{proof}
We compute the spaces and dimensions using the packages for computing quinary
orthogonal modular forms of
G. Rama \cite{rama_quinary}.
The genus of $Q$ has $612$ equivalence classes.
These classes and the $612\times612$ matrix corresponding to
$T_2$ are computed simultaneously via $2$-neighbours,
using the packages of Rama.
Given this (sparse) integral matrix, the subspaces $V_1$ and $V_2$,
and the intersection modulo 5,
are easy to compute using standard linear algebra over the integers
and over $\FF_5$.
The complete code and the corresponding output showing the dimensions
are as stated can be found in~\cite{DT} in two independent versions,
one for Sagemath~\cite{sagemath} and another one for PARI/GP~\cite{pari}.
\end{proof}
The eigenvalues for $T_2$ acting on $F_{79}$ and on $\JR(h_{79})$ are $-5$ and $0$, respectively. We have seen that the oldforms arising from $F_{79}$ account for a $3$-dimensional subspace of $V$, necessarily equal to $V_1$. Since $[F(h_{79}):\Q]=6$, $\JR(h_{79})$ accounts for a $6$-dimensional subspace of $V$, necessarily equal to $V_2$.
Since Hecke operators $T_p$ and $T_{1,p^2}$ (for all $p\neq 5$) act via
their eigenvalues for $F_{79}$ on $V_1$ and via their eigenvalues for
$\JR(h_{79})$ on $V_2$, the existence of this non-zero intersection $V_1\otimes\FF_5\cap V_2\otimes\FF_5$
completes the proof of Theorem~\ref{thm:congruenceJR}.

\section{An outline of the proof of Theorem \ref{main}}

The forms $f_{79}$ and $h_{79}$ may be identified with sections of certain automorphic line bundles on a Hilbert modular surface. Since we wish to show that their Hecke eigenvalues (except for $T_{(\sqrt{5})}$) are the same in $\FF_5$, it is natural to replace them by sections $\overline{f}_{79}$ and $\overline{h}_{79}$ of a Hilbert modular scheme in characteristic $5$. Since $f_{79}$ and $h_{79}$ have different weights,  $\overline{f}_{79}$ and $\overline{h}_{79}$ are sections of different line bundles. To be able to compare them, we need to bring them into the same space. The first-named author mentioned this problem to Hanneke Wiersema, emphasising the difficulty that $\ell$ is ramified in $F$, following her talk at Sheffield Number Theory Day in June 2022. He is indebted to her for directing him to an online lecture by Fred Diamond \cite{D2} (only two months earlier), which provides the generalised partial Hasse invariants (introduced by Reduzzi and Xiao \cite{RX}) and partial theta operators (new in \cite{D1}, \cite{DS}) required. See especially the final slide, which for simplicity assumes total ramification, as in our case. We should add that all that we need is in the cited papers, but the lecture gives a useful summary.

Because $T_{(\sqrt{5})}$ is excluded, it is not quite true to say that it suffices to equate algebraic $q$-expansions of $\overline{f}_{79}$ and $\overline{h}_{79}$, but a partial theta operator $\Theta$ kills the part of the $q$-expansion we are not interested in, so that it suffices to equate $q$-expansions of $\Theta(\overline{f}_{79})$ and $\Theta(\overline{h}_{79})$. After applying $\Theta$, the weights have increased from $(2,2)$ and $(2,4)$ to $(3,3)$ and $(3,5)$. To make the weights the same, we multiply by different monomials in partial Hasse invariants $H_1$ and $H_2$ with $q$-expansion $1$, of weights $(-1,5)$ and $(1,-1)$ respectively. Note that $\Theta, H_1$ and $H_2$ exist only in positive characteristic.

Having arrived at $H_2^3H_1\Theta(\overline{f}_{79})^2$ and $H_2^2\Theta(\overline{h}_{79})^2$, both of the same weight $(8,8)$, we are ready to compare $q$-expansions. The squaring is necessary because $(2,2)$ and $(2,4)$ differ by $(0,2)$, which is not in the lattice spanned by $(-1, 5)$ and $(1, -1)$, whereas $(0,4)=(-1, 5)+(1, -1)$. (There may appear to be a superfluous factor of $H_2^2$, but it is there to make sure that not only are the forms we compare in the same space, but that the common weight is parallel.) Because of the squaring, these elements are not Hecke eigenforms. So even if we could lift them to characteristic zero, we could not apply the Jacquet-Langlands correspondence to replace them by vectors in a space of quaternionic modular forms. Therefore a proof in the style of that of Theorem \ref{thm:congruenceJR} is out of the question.

We aim to prove the equality of $q$-expansions by applying a Sturm bound, directly in characteristic $5$, checking that sufficiently many Fourier coefficients are the same in $\FF_5$. (For this it is better to have the weight parallel.) Assuming appropriate normalisation, we shall actually do this by checking that sufficiently many Hecke eigenvalues of $f_{79}$ and $h_{79}$ are congruent mod $\lambda$. Note that we do not have direct access to Fourier expansions, rather Magma exploits the Jacquet-Langlands correspondence, computing Hecke eigenvalues using either algebraic modular forms for definite quaternion algebras, or the cohomology of Shimura curves. The algorithms used by Magma \cite{BCP} to compute Hecke eigenvalues of Hilbert modular forms, developed by Demb\'el\'e, Donnelly, Greenberg and Voight, are described in \cite{DV}.

Letting $G_1:=H_2^3H_1\Theta(\overline{f}_{79})^2 -H_2^2\Theta(\overline{h}_{79})^2$, we need to show that the $q$-expansion of $G_1$ is zero. For technical reasons to do with reducing the level, we have to replace $G_1$ by a kind of norm $G$, and the weight $(8,8)$ by a multiple $(8J, 8J)$, where $J=1+79=80$ is the index of a certain open compact subgroup of $\GL_2(\OOO_{F,\mathfrak{n}})$. (Recall that $\mathfrak{n}$ is a divisor of $79$.)
The Fourier coefficients are indexed by $\xi\in\dd^{-1}$, where $\dd=(\sqrt{5})$ is the different. Suppose that they vanish for all $\xi$ such that $\tr(\xi)< C$, yet that $G$ is not $0$. Considering the divisor of $G$ shows that a certain divisor $4J(K+D)-CD$ is linearly equivalent to an effective divisor. Here $K$ is a canonical divisor and $D$ is a boundary divisor (resolution of cusps).  Having reduced the level, we happen to be working on a minimal algebraic surface of general type. It follows that $K\cdot(4J(K+D)-CD)\geq 0$. Hence $(C-4J)(K\cdot D)\leq 4J(K\cdot K)$, so 
$$C\leq 4J\left(1+\frac{K\cdot K}{K\cdot D}\right).$$

This bound may be computed on a Hilbert modular surface over $\CC$, and turns out to be $96$. Thus to show that $G=0$ we just need to check that the Fourier coefficients are $0$ for all $\xi\in\dd^{-1}$ such that $\tr(\xi)<97$.
This is achieved by a large computation in Magma. 

Details of the proof we have just sketched are given in the following two sections.

\section{Geometric modular forms in characteristic \texorpdfstring{$\ell$}{ℓ}}

Let $\OOO=\OOO_{F,\sqrt{5}}$, a local completion, with residue field $\FF\simeq \FF_5$. Let $G:=\mathrm{Res}_{F/\QQ}(\GL_2)$, and let $U\subseteq G(\A)\simeq \GL_2(\A_F)$ be the open compact subgroup such that $U_{\pp}=\GL_2(\OOO_{F,\pp})$ for all $\pp\neq 3$ or $\mathfrak{n}$,
$$U_3=\{g\in \GL_2(\OOO_{F, 3}):\,\, g\equiv \begin{pmatrix} *&*\\0&1\end{pmatrix}\pmod{3}\},$$
$$U_{\mathfrak{n}}=\left\{g\in \GL_2(\OOO_{F,\mathfrak{n}}):\,\,g\equiv\begin{pmatrix} *&*\\0&*\end{pmatrix}\pmod{\mathfrak{n}}\right\}.$$ Let $Y_U/\OOO$ be the Pappas-Rapoport model \cite{PR} of the Hilbert modular surface of ``level $U$'', as in \cite[\S 2.2]{D1}. Note that $U$ is ``small enough'', since in \cite[\S 2.2]{DS}, the only CM extension of $F$ under consideration (in their $\mathcal{C}_F$) is $K=\Q(\zeta_5)$, with $r=5$ in their notation, and $\mathfrak{n}$ is inert in $K/F$ because $79\equiv -1\pmod{5}$. Alternatively, we can use the fact that $(3)$ is inert in $K/F$. The choice of $U_3$ also ensures that $-1\notin U$. Similarly if we let $U'$ be such that $U'_{\pp}=\GL_2(\OOO_{F,\pp})$ for all $\pp\neq (3)$,
$$U'_3=\{g\in \GL_2(\OOO_{F, 3}):\,\, g\equiv I\pmod{3}\},$$ then $U'$ is small enough, $-1\notin U'$, and we have a Pappas-Rapoport model $Y_{U'}/\OOO$. 

The Pappas-Rapoport model $Y_U/\OOO$ is a smooth quasi-projective scheme over $\mathrm{Spec}(\OOO)$. To have a smooth special fibre, given that the characteristic $\ell=5$ is ramified in $F/\Q$, the earlier Deligne-Pappas model \cite{DP} would not suffice. 

For each $\kk, \LL\in\Z^2$ with $\kk+2\LL$ of the form $(w,w)$, we have an automorphic line bundle $\AAA_{\kk,\LL, R}$ on the base-change $Y_{U, R}$, where $R$ is any $\OOO$-algebra. (The condition on $\kk+2\LL$ is sufficient to ensure that the descent data referred to prior to \cite[Definition 3.2.1]{D1} exists. Note that $Y_U$ is obtained from some fine moduli scheme by quotienting out by a finite group action involving units.) Then following \cite[Definition 3.2.1]{D1}, $H^0(Y_{U,R}, \AAA_{\kk,\LL,R})$ is denoted $M_{\kk,\LL}(U, R)$, and called the space of Hilbert modular forms of weight $(\kk, \LL)$ and level $U$ with coefficients in $R$. We shall use $(\kk,\LL)=((2,2),(0,0))$ and $((2,4),(1,0))$, among others. It is primarily $\kk$ we are interested in, but $\LL$ affects the Hecke action. It is important that the level is the full $\GL_2(\OOO_{F,\pp})$ at the unique prime $\pp=(\sqrt{5})$ sharing the characteristic $\ell=5$ of the special fibre.

The weight-$(2,2)$ Hecke eigenform $f_{79}$ may be viewed as an element of \newline $M_{(2,2),(0,0)}(U, \CC)$. It has a $q$-expansion 
$$f_{79}(z_1,z_2)=\sum_{\xi\in\dd^{-1},\,\xi\gg 0}a_{\xi}e^{2\pi i(\xi_1z_1+\xi_2z_2)}$$ at the cusp $(i\infty, i\infty)$, where the sum is over totally positive elements $\xi$ of the inverse different, with $\xi_1, \xi_2$ the images of $\xi$ under the embeddings of $F$ into $\R$. 

At this point we need to enlarge $\OOO$ and $\FF$, to accommodate $N^{\mathrm{th}}$-roots of unity, where $N$ is such that $U$ contains the principal congruence subgroup 
$$U(N):=\{g\in  \GL_2(\A_F):\,\, g\equiv \begin{pmatrix} 1&0\\0&1\end{pmatrix}\pmod{N}\},$$
and for us $N=3\times 79$. Cusps and $q$-expansions may be defined purely algebro-geometrically, as in \cite[\S 7]{D1}. Given $f\in M_{\kk,\LL}(U, R)$, with $R$ an $\OOO$-algebra, at each cusp of $Y(U)$ there is a $q$-expansion
$$\sum_{m\in N^{-1}M_{+}\cup\{0\}}b\otimes r_m q^m$$ as in \cite[Proposition 7.2.1]{D1}, where $b$, depending on the cusp, is a basis for a certain line bundle, all the $r_m$ belong to $R$, $M$ is as in \cite[(17)]{D1}, and $q^m$ is a formal monomial. We now state, in the setting of \cite{D1}, a part of the $q$-expansion principle not covered by \cite[Proposition 7.2.2]{D1}, but the conclusion is of the same nature as \cite[Theoreme 6.7(ii)]{Ra}, and may be proved the same way. 
\begin{lem}\label{qexp}
Let $R\subset S$ be $\OOO$-algebras and $f\in M_{\kk,\LL}(U, S)$. Let $\Sigma$ be a set of cusps containing one attached to each geometric component of $Y(U)$. Suppose that for each cusp in $\Sigma$, the $q$-expansion coefficients $r_m$ (a priori in $S$) all belong to $R$. Then $f$ arises via base change from an element of $M_{\kk,\LL}(U, R)$. (In general here, $U\supseteq U(N)$ is ``small enough'', with $\ell\nmid N$, and $\kk+2\LL$ is of the form $(w,w)$.)
\end{lem}

We apply this to our case, with $f=f_{79}$ and $S=\CC$. Since $\det(U)$ is the whole of $\widehat{\OOO_F}^{\times}$ and $\OOO_F$ has narrow class number $1$, $Y_U$ has a single geometric component \cite[end of \S 2.2]{D1}. For us, $M=\dd^{-1}$, $m$ may actually be taken in $M_{+}\cup\{0\}$ (omitting the factor $N^{-1}$), as at the end of \cite[\S 7.2]{D1}, and the $r_m$ are our $a_{\xi}$. We may normalise the newform $f_{79}$ in such a way that all the $a_{\xi}$ are Hecke eigenvalues (for $T_{(\xi)\dd}$), hence lie in $\ZZ$ so in $\ZZ_5\subset\OOO$.

It follows, from the fact that the Fourier coefficients of $f_{79}$, at a cusp on each geometric component, are not just in $\CC$ but in $\OOO$, that $f_{79}\in M_{(2,2),(0,0)}(U, \OOO)$. It then has a base change $\overline{f}_{79}\in M_{(2,2),(0,0)}(U, \FF)$. Similarly, using the embedding of $F(h_{79})$ in $\Q_5$ corresponding to the degree-$1$ divisor $\lambda$, we get $\overline{h}_{79}\in M_{(2,4),(1,0)}(U, \FF)$.

There are natural actions of Hecke operators ($T_{\pp}$ for $\pp\neq (\sqrt{5})$) on the $M_{\kk,\LL}(U, R)$, and we would like to show that the Hecke eigenvalues of $\overline{f}_{79}$ and $\overline{h}_{79}$ are the same in $\FF$. 
Let $$f_{79}=\sum_{\xi\in\dd^{-1},\,\xi\gg 0}a_{\xi}e^{2\pi i(\xi_1z_1+\xi_2z_2)}$$ and $$h_{79}=\sum_{\xi\in\dd^{-1},\,\xi\gg 0}b_{\xi}e^{2\pi i(\xi_1z_1+\xi_2z_2)}$$ be $q$-expansions at the cusp $(i\infty, i\infty)$.
To prove that the Hecke eigenvalues of $\overline{f}_{79}$ and $\overline{h}_{79}$ (of $T_{\pp}$ for $\pp\neq (\sqrt{5})$) are the same in $\FF$, it suffices to show that $a_{\xi}\equiv b_{\xi}\pmod{\lambda}$ for all $\xi$ with $(\xi)\dd$ coprime to $5$, i.e. for $(\sqrt{5})\nmid(\xi)\dd$. Note that this congruence does not hold without the condition $(\sqrt{5})\nmid(\xi)\dd$, since the eigenvalues of $T_{(\sqrt{5})}$ on $f_{79}$ and $h_{79}$ are not congruent mod $\lambda$, as one can see from the table in \S \ref{sect:HMF} above. 

Following \cite[\S 4.1]{D1}, we have generalised partial Hasse invariants \newline $H_1\in M_{(-1,\ell),(0,0)}(U, \FF)$ and $H_2\in M_{(1,-1),(0,0)}(U, \FF)$, well-defined up to $\FF^{\times}$ multiples, and their product $H:=H_1H_2\in M_{(0,\ell -1),(0,0)}(U, \FF)$. In general the weights would be $(-1,0,\ldots,0,\ell)$,
$(1,-1,0,\ldots,0)$, $(0,1,-1,0,\ldots,0),\ldots,$$(0,\ldots,0,1,-1)$, where the number of entries is the ramification index, and for simplicity we assume total ramification. We also have a partial theta operator $\Theta: M_{\kk,\LL}(U, \FF)\rightarrow M_{\kk+(1,1),\LL-(0,1)}(U, \FF)$, as in \cite[\S 5.2]{D1}, but beware the typo ``$\LL+(0,1)$'' compared to \cite[Introduction]{D1}. There is only one theta operator, since there is only one embedding from $\OOO_F/(\sqrt{5})$ to $\Fbar_5$.

First we apply the partial theta operator $\Theta$ to both $\overline{f}_{79}$ and $\overline{h}_{79}$. We may look at \cite[(30), \S 8.2]{D1} for the effect of $\Theta$ on $q$-expansions, now in characteristic $\ell=5$, so the exponentials are just formal algebraic entities. A coefficient $a_{\xi}$ gets multiplied by the image of $\xi\delta$ in $\FF$ (up to some technicalities about local trivialisations of line bundles), where $\delta$ is a totally positive generator of $\dd$. (This reflects the fact that $\Theta$ involves some kind of differentiation.) In particular, $\Theta$ conveniently kills all the Fourier coefficients for $(\sqrt{5})\mid (\xi)\dd$. In fact, what we want is equivalent to {\em all} the Fourier coefficients of $\Theta(\overline{f}_{79})\in M_{(3,3),(0,-1)}(U, \FF)$ and  $\Theta(\overline{h}_{79})\in M_{(3,5),(1,-1)}(U, \FF)$ being the same in $\FF$. Let $\mathcal{F}$ and $\mathcal{G}$ be the Fourier expansions of $\Theta(\overline{f}_{79})$ and $\Theta(\overline{h}_{79})$, respectively.

Now $H_1\in M_{(-1,5),(0,0)}(U, \FF)$ and $H_2\in M_{(1, -1),(0,0)}(U, \FF)$ have constant $q$-expansions \cite[\S 8.1]{D1}, which we may choose to be $1$ by choice of scaling. Thus multiplication by any monomial in $H_1$ and $H_2$ has no effect on $q$-expansions. In particular, $H_2\Theta(\overline{h}_{79})\in M_{(4,4),(1,-1)}(U, \FF)$ also has $q$-expansion $\mathcal{G}$.

The ring of formal $q$-expansions over $\FF$ is a subring of a power series ring with coefficients in a ring of Laurent series over $\FF$. As such, it is an integral domain.
Since $$\mathcal{F}^2-\mathcal{G}^2=(\mathcal{F}-\mathcal{G})(\mathcal{F}+\mathcal{G}),$$ and it is easy to check that the second factor is not zero, it suffices to check that the Fourier expansions of $\Theta(\overline{f}_{79})^2\in M_{(6,6),(0,-2)}(U, \FF)$ and  $(H_2\Theta(\overline{h}_{79}))^2\in M_{(8,8),(2,-2)}(U, \FF)$ are the same. 

Proving this is equivalent to showing that $H_2^3H_1\Theta(\overline{f}_{79})^2\in M_{(8,8),(0,-2)}(U, \FF)$ and $H_2^2\Theta(\overline{h}_{79})^2\in M_{(8,8),(2,-2)}(U, \FF)$, have the same $q$-expansion. Following \cite[\S 3.4]{DS}, $\AAA_{(0,0),(0,2)}$ and $\AAA_{(0,0),(2,0)}$ are trivial bundles, because every totally positive unit $u\in \OOO_F^{\times}$ is a square, so satisfies $u^2\equiv 1\pmod{\sqrt{5}}$ (for either embedding of $u$). Hence we may view the desired equality of $q$-expansions as between two elements of $M_{(8,8),(-4,-4)}(U, \FF)$. The reason for this choice will become clear later.

Recall that we aim to prove this equality by applying a Sturm bound, directly in characteristic $\ell$, checking that sufficiently many Fourier coefficients are the same in $\FF$. Also that, with appropriate normalisation, we shall do this by checking that sufficiently many Hecke eigenvalues of $f_{79}$ and $h_{79}$ are congruent mod $\lambda$. 

In \cite[\S 5]{BP}, Burgos Gil and Pacetti obtain a Sturm bound in the case of parallel weights. (They also outline, in their \S 6, how to generalise it to non-parallel weights, by multiplying a form $G(z_1,z_2)$ by $G(z_2,z_1)$ to attain parallel weight, but for us, our earlier multiplication by $H_2^2$ reaches a smaller parallel weight, hence a better bound.) They work under the condition that $\ell$ is not ramified in $F$, and use Rapoport's toroidal compactification \cite{Ra} of the Deligne-Pappas model. We shall follow the steps of their argument closely, but since for us $\ell$ is ramified in $F$, we use the Pappas-Rapoport model instead.

\begin{lem}\label{coeffs}
Let $G_j\in M_{\kk_j,\LL_j}(U, \FF)$ for $j=1,2,3$, with $G_3=G_1G_2$ and $(\kk_3,\LL_3)=(\kk_1,\LL_1)+(\kk_2,\LL_2)$. Fix an element $A\in F$ with $A\gg 0$, i.e. $A$ is totally positive. Suppose that $G_j$ has $q$-expansion $\sum_{\xi\in\dd^{-1},\,\xi\gg 0}b_{\xi, j}e^{2\pi i(\xi_1z_1+\xi_2z_2)}$ at the cusp $(i\infty, i\infty)$ (or rather at the corresponding boundary component of the minimal compactification of $Y_U$, as in \cite[\S 7.2]{D1}). If $b_{\xi,1}=0$ for all $\xi\in\dd^{-1}$ with $\xi\gg 0$ and $\tr(A\xi)<B$, for some $B$ not specified here, then $b_{\xi, 3}=0$ for all $\xi\in\dd^{-1}$ with $\xi\gg 0$ and $\tr(A\xi)< B$.
\end{lem}
\begin{proof} We have $b_{\xi,3}=\sum_{\eta} b_{\eta, 1}b_{\xi-\eta, 2}$, where $\eta\in\dd^{-1}$ is required to be such that $\eta\gg 0$ and $\xi-\eta\gg 0$. Since $A(\xi-\eta)\gg 0$, $\tr(A\eta)<\tr(A\xi)$. If $\tr(A\xi)< B$ then $\tr(A\eta)<B$ so $b_{\eta, 1}=0$, for all $\eta$ in the sum, hence $b_{\xi, 3}=0$, as required.
\end{proof}

Now we let $$G_1=H_2^3H_1\Theta(\overline{f}_{79})^2-H_2^2\Theta(\overline{h}_{79})^2\in M_{(8,8),(-4,-4)}(U, \FF).$$ Also let $J=[\GL_2(\OOO_{F,\mathfrak{n}}): U_{\mathfrak{n}}]$, and $\{g_i: \,1\leq i\leq J\}$ a set of representatives of the cosets of $U_{\mathfrak{n}}$ in $\GL_2(\OOO_{F,\mathfrak{n}})$, with $g_1=\mathrm{id}$. Each $g_i$ may be thought of as an element of $\GL_2(\A_F)$, and gives $[g_i]G_1\in  M_{(8,8),(-4,-4)}(g_iUg_i^{-1},\FF)$, with $[g]$ as in \cite[\S 3.2]{D1}.
The product $G:=\prod_{i=1}^J([g_i]G_1)\in M_{(8J,8J),(-4J,-4J)}(\cap_{i=1}^Jg_iUg_i^{-1},\FF)$ is invariant under $U''$, defined like $U$ except with $U''_{\mathfrak{n}}=\GL_2(\OOO_{F,\mathfrak{n}})$, hence under $U'\subset U''$. (Recall that 
$$U''_3=U_3=\{g\in \GL_2(\OOO_{F, 3}):\,\, g\equiv \begin{pmatrix} *&*\\0&1\end{pmatrix}\pmod{3}\},$$
$$U'_3=\{g\in \GL_2(\OOO_{F, 3}):\,\, g\equiv I\pmod{3}\}.)$$
It follows that we may view $G\in M_{(8J,8J),(-4J,-4J)}(U',\FF)$. Let $$G=\sum_{\xi\in\dd^{-1},\,\xi\gg 0}c_{\xi}e^{2\pi i(\xi_1z_1+\xi_2z_2)}.$$
Repeated application of Lemma \ref{coeffs} (in which the choice of $U$ is not important) shows that if $H_2^3H_1\Theta(\overline{f}_{79})^2$ and $H_2^2\Theta(\overline{h}_{79})^2$ have the same Fourier coefficients in $\FF$, for all $\xi\in\dd^{-1}$ with $\xi\gg 0$ and $\tr(A\xi)< B$ (so that the corresponding Fourier coefficients of $G_1$ vanish), then $c_{\xi}=0$ in $\FF$ for all $\xi\in\dd^{-1}$ with $\xi\gg 0$ and $\tr(A\xi)< B$, where $A\gg 0$. Our goal is (for suitably chosen $A$) to make $B$ large enough that we can deduce that $G=0$. Since the ring of formal $q$-expansions with coefficients in $\FF$ is an integral domain, as explained above, it would follow that $G_1=0$, as required. 

\section{Intersection numbers on a Hilbert modular surface}
We have $G\in H^0(Y_{U',\FF}, \AAA_{(8J,8J),(-4J,-4J),\FF})$, and would like to show that $G=0$. We may choose a toroidal compactification $Y^{\mathrm{tor}}_{U'}/\OOO$, by \cite[\S 2.11]{RX}. Then $\AAA_{(8J,8J),(-4J,-4J)}$ extends to $ \AAA^{\mathrm{tor}}_{(8J,8J),(-4J,-4J)}$, and $$H^0(Y_{U',\FF}, \AAA_{(8J,8J),(-4J,-4J),\FF})\simeq H^0(Y^{\mathrm{tor}}_{U',\FF}, \AAA^{\mathrm{tor}}_{(8J,8J),(-4J,-4J),\FF}).$$

As in \cite[Theorem 2.9]{RX}, \cite[\S 3.3]{D1}, there is a Kodaira-Spencer isomorphism of line bundles on $Y_{U'}$, between $\AAA_{(2,2),(-1,-1)}$ and the canonical sheaf $\mathcal{K}=\wedge^2\Omega$. Following \cite[\S 7.3]{D1}, this extends to an isomorphism between $\mathcal{K}^{\mathrm{tor}}(\log(D))$ (i.e. $\mathcal{K}^{\mathrm{tor}}\otimes\OOO_{Y^{\mathrm{tor}}_{U'}}(D)$) and $\AAA^{\mathrm{tor}}_{(2,2),(-1,-1)}$, hence between $\mathcal{K}^{\mathrm{tor}}$ and $\AAA^{\mathrm{tor}}_{(2,2),(-1,-1)}(-D)$ (i.e. $\AAA^{\mathrm{tor}}_{(2,2),(-1,-1)}\otimes\OOO_{Y^{\mathrm{tor}}_{U'}}(-D)$), where $D=Y^{\mathrm{tor}}_{U'}\backslash Y_{U'}$ is a divisor with simple normal crossings. 

Hence $G\in H^0(Y^{\mathrm{tor}}_{U',\FF}, \AAA^{\mathrm{tor}}_{(8J, 8J),(-4J,-4J),\FF})$ may be identified with an element of $H^0(Y^{\mathrm{tor}}_{U',\FF}, (\mathcal{K}^{\mathrm{tor}}(\log(D)))^{\otimes 4J})$. Thus the divisor of $G$, viewed as a section of the line bundle $(\mathcal{K}^{\mathrm{tor}}(\log(D)))^{\otimes 4J}$, is in the linear equivalence class of $4J(K+D)$, where $K$ is a canonical divisor. (We assume, for a contradiction, that $G\neq 0$.) Note that $K$ and $D$ are divisors on $Y^{\mathrm{tor}}_{U'}/\OOO$, but we are looking at their restrictions to the special fibre, and using the same notation. Now suppose that the order of $G$ at each irreducible component of $D$ is at least $C$. In other words, $4J(K+D)-CD$ is linearly equivalent to an effective divisor. 

\begin{lem}\label{Cleq96}
If $G\neq 0$ and the order of $G$ at each irreducible component of $D$ is at least $C$, then $C\leq 96$.
\end{lem}
\begin{proof}
The Hilbert modular surface $Y_{U'}(\C)$ has a Satake compactification obtained by adding $T$ ``cusps'', each lying over the orbit of $\infty=(i\infty, i\infty)$ in $\SL_2(\OOO_F)\backslash (\HH^2\cup \PP^1(F))$. In \cite[(6)]{BP}, $d=\frac{1}{2}\#\SL_2(\OOO_F/3\OOO_F)=360$, $n^2=9$, $[U^2_{\OOO_F}: U^2_{\OOO_F, (3)}]=\frac{1}{2}\#(\OOO_F/(3))^{\times}=4$ and $c'=T$, so $d=n^2c'[U^2_{\OOO_F}: U^2_{\OOO_F, (3)}]\implies T=10$. Resolving the cusp singularities as in \cite{vdG}, one obtains a nonsingular algebraic surface $Y/\C$. As noted in \cite[Introduction]{Ra}, toroidal compactification recovers, over $\C$, this desingularisation process. Hence we may assume that $Y\simeq Y^{\mathrm{tor}}_{U'}(\C)$. 

By \cite[Appendix B]{BP}, which cites \cite[Example 7.5]{vdG}, our $Y$ is a minimal algebraic surface of general type. (This motivated our choice of $U'$. In their notation, $D=5$ and $\mathfrak{c}=3$.) The boundary divisor $D$ (viewed over $\C$) is composed of the resolutions of $T=10$ cusps, each a cycle of $4$ copies $S_1,\ldots, S_4$ of $\PP^1$, each self-intersection number $S_i^2=-3$. By \cite[Theorem 9.1, Lemma 9.6]{KU}, the geometric special fibre of $Y^{\mathrm{tor}}_{U'}$ is also a minimal algebraic surface of general type. Consequently, by \cite[Proposition 1]{B}, the fact that $4J(K+D)-CD$ is linearly equivalent to an effective divisor implies that $K\cdot(4J(K+D)-CD)\geq 0$. Hence $(C-4J)(K\cdot D)\leq 4J(K\cdot K)$, so 
$$C\leq 4J\left(1+\frac{K\cdot K}{K\cdot D}\right).$$

By \cite[Lemma 9.3]{KU}, these intersection numbers on the geometric special fibre may be computed on $Y/\C$. According to \cite[(12)]{BP}, and in their notation,
$$K\cdot K=4d\zeta_F(-1)+\frac{d}{n^2}\sum_{i=1}^h\sum_j(2-b_{i,j}),$$
while by \cite[(11)]{BP},
$$K\cdot D=\frac{d}{n^2}\sum_{i=1}^h\sum_j(b_{i,j}-2).$$
Dividing one by the other, 
$$1+\frac{K\cdot K}{K\cdot D}=\frac{4n^2\zeta_F(-1)}{\sum_{i=1}^h\sum_j(b_{i,j}-2)},$$
so we find that 
\begin{equation}\label{Cbound}
C\leq J\frac{16n^2\zeta_F(-1)}{\sum_{i=1}^h\sum_j(b_{i,j}-2)}.
\end{equation}
For us, $n^2=9$, $h=1$ and $j$ runs from $1$ to $4$, with each $b_{i,j}=-S_j^2=3$. Also $\zeta_F(-1)=1/30$, so 
$$C\leq \frac{6J}{5}.$$
Since $$J=[\GL_2(\OOO_{F,\mathfrak{n}}) : U_{\mathfrak{n}}]=\#\PP^1(\FF_{79})=(1+79)=80,$$ this becomes a bound of 
\begin{equation}\label{OurCBound} C\leq 96. 
\end{equation} 
\end{proof} 
To prove that $G=0$, hence to complete the proof of Theorem \ref{main}, we need to show that $C\geq 97$. 

\begin{lem}\label{Cgeq97} 
If $G\neq 0$ then $\divv(G)\geq 97D$, i.e. $C\geq 97$.
\end{lem}
\begin{proof} Recall that $C$ is the minimum of the orders of $G$ at the irreducible components of $D$, the cuspidal resolution divisor. Let $D=\sum_{i=1}^{10}D_i$, a sum of the resolution divisors associated to the $10$ cusps of the Satake compactification of $Y_{U'}(\C)$. There are $10$ cusps rather than $1$ only because we introduced the auxiliary level at $3$, whereas $f_{79}$ and $h_{79}$ come from $\Gamma_0(\mathfrak{n})$, which has the single cusp $(i\infty, i\infty)$.
Therefore if we could show that $\divv(G)\geq 97D_1$, where $D_1$ corresponds to $(i\infty, i\infty)$, it would follow that $\divv(G)\geq 97D$, as required.

To do this, we use \cite[Lemma 3.3]{BP}. Following the procedure in \cite[Appendix A]{BP}, we find the sequence $(A_j)_{j\in\Z}$ referred to in that Lemma. It is $A_j=\left(\frac{3+\sqrt{5}}{2}\right)^{-j}$. In particular, all the $A_j$ are squares of units, so $c_{A_j\xi}=c_{\xi}$ for any $\xi\in\dd^{-1}$ with $\xi\gg 0$.  According to \cite[Lemma 3.3]{BP}, we need to show that $c_{\xi}=0$ (in $\FF$) for any $\xi\in\dd^{-1}$ with $\xi\gg 0$, such that $\tr(A_j\xi)<97$ for some $j$. The condition becomes simply that $c_{\xi}=0$ for all $\xi\in\dd^{-1}$, with $\xi\gg 0$, such that $\tr(\xi)<97$. 

Recall that
$$f_{79}=\sum_{\xi\in\dd^{-1},\,\xi\gg 0}a_{\xi}e^{2\pi i(\xi_1z_1+\xi_2z_2)}$$ and $$h_{79}=\sum_{\xi\in\dd^{-1},\,\xi\gg 0}b_{\xi}e^{2\pi i(\xi_1z_1+\xi_2z_2)}.$$ By the final paragraph of the previous section, we just need to show that $a_{\xi}\equiv b_{\xi}\pmod{\lambda}$ whenever $(\sqrt{5})\nmid(\xi)\dd$ and $\tr(\xi)\leq 96$. Since $a_{\xi}$ and $b_{\xi}$ are the eigenvalues of $T_{\dd(\xi)}$ on $f_{79}$ and $h_{79}$ respectively, it suffices to show that their eigenvalues for $T_{\pp}$ (say $\mu_{\pp}(f_{79})$ and $\mu_{\pp}(h_{79})$) are the same in $\FF$, for all prime ideals $\pp\neq (\sqrt{5})$ of $\OOO_F$ dividing elements of the form $\sqrt{5}\xi$ with  $\xi\in\dd^{-1}$, $\xi\gg 0$ and $\tr(\xi)<97$. 

We got Magma to collect a list of such $\pp$, of which there are $1313$. For each one we got it to compute the Hecke eigenvalues $\mu_{\pp}(f_{79})$ and $\mu_{\pp}(h_{79})$. For the largest $\pp$ each of these takes around $50$ seconds. Then we got it to display the images of these eigenvalues in $\FF_5$. The code and the output may be found at \cite{DT}, and looking at the output, the congruence is visibly satisfied for all the $\pp$ on the list.
\end{proof}
The contradiction between Lemmas \ref{Cleq96} and \ref{Cgeq97} proves that $G=0$, completing the proof of Theorem \ref{main}.

\section{A mod 2 congruence}
The rational prime $2$ is inert in $F=\Q(\sqrt{5})$. In the cubic extension $F(h_{79})/F$ it factorises as $(2)=\qq_1^2\qq_2$ with each $\qq_i$ of degree $1$ over $F$, absolute norm $4$. In the above computations, we also checked $\mu_{\pp}(f_{79})$ and $\mu_{\pp}(h_{79})$ mod $\qq_1$. We found that for all $\pp\neq (2)$ or $\nn$ (and on the list) we have $\mu_{\pp}(f_{79})\equiv\mu_{\pp}(h_{79})\equiv 0\pmod{\qq_1}$, while $\mu_{\nn}(f_{79})=\mu_{\nn}(h_{79})=1$, $\mu_{(2)}(f_{79})=1$ and $\mu_{(2)}(h_{79})\equiv 0\pmod{\qq_1}$.

The observations about $f_{79}$ are entirely accounted for by the fact that the associated elliptic curves over $F$, in the isogeny class \lmfdbecnf{2.2.5.1}{79.1}a{}, have $F$-rational $2$-torsion points, in fact for \lmfdbecnf{2.2.5.1}{79.1}a2 the $2$-torsion is entirely $F$-rational. Thus the representation of $\Gal(\Qbar/F)$ on the $2$-torsion is trivial, with character value $1+1=2\equiv 0\pmod{\qq_1}$ for every element. This includes Frobenius elements for all $\pp\neq (2),\nn$, the primes for which the $\qq_1$-adic representation is unramified, so that a trace of Frobenius exists and is equal to the Hecke eigenvalue $\mu_{\pp}(f_{79})$.

To extend this to an explanation of what is observed for $h_{79}$, we need the following.

\begin{thm}\label{mainmod2}
$$\mu_{\pp}(f_{79})\equiv\mu_{\pp}(h_{79})\pmod{\qq_1},$$
for all primes $\pp\neq(2)$ in $F$.
\end{thm}
\begin{proof} One argues very much as in the proof of Theorem \ref{main}. Now $\FF$ will be $\FF_4$ rather than $\FF_5$. Since $2$, unlike $5$, is inert in $F$, things are a little different. We have $H_1$ and $H_2$ now of weights $\kk=(-1,\ell)=(-1,2)$ and $(\ell,-1)=(2,-1)$, respectively, both with trivial $\LL$. Since there are two embeddings of $\OOO_F/(2)$ in $\Fbar_2$, there are now two partial theta operators, $\Theta_1: M_{\kk,\LL}(U, \FF)\rightarrow M_{\kk+(1,2),\LL-(1,0)}(U, \FF)$ and $\Theta_2: M_{\kk,\LL}(U, \FF)\rightarrow M_{\kk+(2,1),\LL-(0,1)}(U, \FF)$ (where $2$ is $\ell$), as in \cite[\S 5.2]{D1}. These are essentially the partial Hasse invariants and theta operators of Andreatta and Goren \cite{AG}.

Starting with $\overline{f}_{79}$ and $\overline{h}_{79}$, of weights $\kk=(2,2)$ and $(2,4)$ respectively, first we apply $\Theta_1$ to both sides, to get weights $(\kk,\LL)=((3,4),(-1,0))$ and $((3,6),(0,0))$. According to  \cite[Theorem A]{D1}, since the (first) components of $(2,2)$ and $(2,4)$ are divisible by $\ell=2$, both $\Theta_1(\overline{f}_{79})$ and   $\Theta_1(\overline{h}_{79})$ are divisible by $H_1$. Dividing both by $H_1$ produces weights $(\kk,\LL)=((4,2),(-1,0))$ and $((4,4),(0,0))$. Since $u^3\equiv 1\pmod{(2)}$ for any (totally positive) $u\in \OOO_F^{\times}$, $\AAA_{(0,0),(0,3)}$ and $\AAA_{(0,0),(3,0)}$ are trivial bundles \cite[\S 3.4]{DS}. 
Therefore the weights $(\kk,\LL)=((12,12),(-3,0))$ and $((12,12),(0,0))$ we get after cubing and multiplying by $H_1^4H_2^2$, or just cubing, respectively, may be converted to $((12,12),(-6,-6))$.

Multiplying the original bound $C\leq 96$ by $12/8 $, it becomes now $C\leq 144$. So the earlier computation needs to be extended to $97\leq \tr(\xi)< 145$. The results, confirming the congruence, may be found in the output file in \cite{DT}.
\end{proof}
\begin{cor} $\mu_{\pp}(h_{79})\equiv 0\pmod{\qq_1}$ for all $\pp\neq (2), \nn$.
\end{cor}

{\bf Acknowledgments. } The main question addressed here came out of online discussions with Vasily Golyshev, Duco van Straten and others. Hanneke Wiersema, as noted earlier in more detail, directed the first-named author to some of the tools necessary to prove Theorem \ref{main}. We are grateful to two anonymous referees for their careful reading and valuable comments.

\end{document}